\documentclass[12pt,a4paper]{amsart}
\usepackage{amsmath, amssymb, amsfonts,enumerate}
\usepackage[all]{xy}
\usepackage{amscd}
\usepackage{comment}

\newtheorem{theorem}{Theorem}[section]
\newtheorem{lemma}[theorem]{Lemma}
\newtheorem{corollary}[theorem]{Corollary}

 \theoremstyle{definition}
 
 \newtheorem{remark}[theorem]{Remark}

 \newtheorem{example}[theorem]{Example}

\numberwithin{equation}{section}
\newcommand {\N}{\mathbb{N}} 
\newcommand {\Z}{\mathbb{Z}} 
\newcommand {\R}{\mathbb{R}} 
\newcommand {\Q}{\mathbb{Q}} 
\newcommand {\C}{\mathbb{C}} 

\newcommand{\T}{\mathbb{T}}

\begin{document}
\title[Surjunctivity of algebraic dynamical systems]{A note on the surjunctivity of algebraic dynamical systems}
\author{Tullio Ceccherini-Silberstein}
\address{Dipartimento di Ingegneria, Universit\`a del Sannio, C.so
Garibaldi 107, 82100 Benevento, Italy}
\email{tceccher@mat.uniroma3.it}
\author{Michel Coornaert}
\address{Universit\'e de Strasbourg, CNRS, IRMA UMR 7501, F-67000 Strasbourg, France}
\email{michel.coornaert@math.unistra.fr}
\begin{abstract}
Let $X$ be a compact  metrizable  group and  $\Gamma$  a countable group acting  on $X$ by continuous group automorphisms.
We give sufficient conditions under which the dynamical system $(X,\Gamma)$ is  surjunctive, i.e.,
every injective continuous map $\tau \colon X \to X$ commuting with the action of 
$\Gamma$  is surjective.
\end{abstract}
\date{\today}
\maketitle

\section{Introduction}

Consider a dynamical system $(X,\Gamma)$ consisting of a compact metrizable space $X$ equipped with a continuous action of a countable group $\Gamma$.
A map $\tau \colon X \to X$ is said to be \emph{$\Gamma$-equivariant} if it commutes with the action of $\Gamma$ on $X$, i.e., 
$\tau(\gamma x) = \gamma \tau(x)$ for all $\gamma \in \Gamma$ and $x \in X$.
Following a terminology introduced by Gottschalk in \cite{gottschalk},
let us say that the dynamical system $(X,\Gamma)$ is \emph{surjunctive} if every 
injective $\Gamma$-equivariant continuous map
$\tau \colon X \to X$ is surjective (and hence a homeomorphism).
\par
There are several important classes of dynamical systems that are known to be surjunctive.
\par
First note that the dynamical system  $(X,\Gamma)$ is surjunctive whenever  the phase pace $X$ is \emph{incompressible}, i.e., there is no proper subset of $X$ 
that is  homeomorphic to $X$.
This is for example the case when $X$ is a closed topological manifold (e.g. a compact Lie group) by Brouwer's invariance of domain.
\par
Another class of surjunctive dynamical systems is provided  by the systems that satisfy the descending chain condition.
We recall that one says that the dynamical system  $(X,\Gamma)$ satisfies the \emph{descending chain condition} (\emph{d.c.c.}\ for short)  if every decreasing sequence
$$
X = X_0 \supset X_1 \supset X_2 \supset \dots
$$
of $\Gamma$-invariant closed subsets $X_i \subset X$ ($i \in \N$) eventually stabilizes, i.e., there is an integer $k \geq 0$ such that
$X_i = X_k$ for all $i \geq k$.
Indeed, suppose that $(X,\Gamma)$ satisfies the d.c.c.\ and $\tau \colon X \to X$ is an injective continuous $\Gamma$-equivariant map.
  By applying the d.c.c.\ to the sequence
  $$
  X= \tau^0(X) \supset \tau(X) = \tau^1(X) \supset \tau^2(X) \supset \dots,
  $$
  we see that there is an integer $k \geq 0$ such that
  $\tau^k(X) = \tau^{k + 1}(X)$.
Since $\tau^k$ is injective, this implies 
$X = \tau(X)$, showing that $\tau$ is surjective.
  Note that a minimal dynamical system
  (i.e., a system containing no proper invariant closed subsets)
   trivially satisfies the d.c.c.\
  Actually every $\Gamma$-equivariant continuous map $\tau \colon X \to X$ is surjective when 
  $(X,\Gamma)$ is minimal.
  More generally, $(X,\Gamma)$ satisfies the d.c.c.\ if $X$ contains only finitely many closed 
  $\Gamma$-invariant subsets or if every closed $\Gamma$-invariant subset of $X$ is finite. 
  \par
There are also several classes of symbolic dynamical systems that are known to be surjunctive. Let $S$ be a  compact metrizable 
space, called the \emph{alphabet} or the \emph{space of symbols}.
   Given a countable group $\Gamma$,  the 
  $\Gamma$-\emph{shift} on the alphabet $S$ is the dynamical system   $(S^\Gamma, \Gamma)$,
  where $S^\Gamma = \{x \colon \Gamma \to S\}$ is equipped with the product  topology  and the action of $\Gamma$ on 
  $S^\Gamma$ is given by the formula $(\gamma x)(\gamma') := x(\gamma^{-1} \gamma')$ for 
  all $\gamma,\gamma' \in \Gamma$ and $x \in S^\Gamma$. 
These shift systems are not surjunctive in general.
Indeed, suppose  for example that the alphabet space $S$ is compressible (e.g. $S$ is the closed unit interval $[0,1] \subset \R$
or the Cantor set $K \subset [0,1]$ or the infinite-dimensional torus
$\T^\N$ with $\T := \R/\Z$) and let $\iota \colon S \to S$ be an injective continuous map that is not surjective. 
Then the map $\tau \colon S^\Gamma \to S^\Gamma$ given by
$\tau(x)(\gamma) := \iota(x(\gamma))$ for all $x \in S^\Gamma$ and $\gamma \in \Gamma$,
is clearly   injective, $\Gamma$-equivariant, and continuous.
However, $\tau$ is  not surjective since $\iota$ is not.
Thus, the shift system $(S^\Gamma,\Gamma)$ is never surjunctive when the alphabet  $S$ is compressible.
If $M$ is a closed topological manifold and the group $\Gamma$ is residually finite, 
it was observed in \cite[Corollary~7.8]{csc-concrete} that  the 
$\Gamma$-shift
$(M^\Gamma,\Gamma)$ is surjunctive.
 On the other hand, it follows from a deep theorem of Gromov~\cite{gromov-esav} and 
 Weiss~\cite{weiss-sgds} that,   if $\Gamma$ is a sofic group and $S$ is a finite discrete space, then
 the $\Gamma$-shift  $(S^\Gamma,\Gamma)$  is surjunctive.
Sofic groups form a wide class of groups containing in particular all residually amenable  groups
 but it is not known if the Gromov-Weiss surjunctivity theorem remains valid  for all groups (\emph{Gottschalk conjecture}). Actually, there is no example of a non-sofic group up to now
 although it seems that many experts in the field do believe in the existence of non-sofic groups.
 \par
 One says that a dynamical system  $(X,\Gamma)$ is \emph{expansive} if there exists a constant 
$\delta = \delta(X,\Gamma,d)  > 0$ such that, for every pair of distinct points $x,y \in X$, there exists an element
$\gamma = \gamma(x,y)  \in \Gamma$ such that
$d(\gamma  x,\gamma  y) \geq  \delta$. Here $d$ denotes a compatible metric on $X$.
The fact that this  definition does not depend on the choice of $d$
follows from the  compactness of $X$.
For instance, a shift system $(S^\Gamma,\Gamma)$ is expansive if and only if the alphabet $S$ is finite.
\par
Given a dynamical system $(X,\Gamma)$, a point $x \in X$ is said to be $\Gamma$-\emph{periodic}
if its $\Gamma$-orbit $\Gamma x :=  \{\gamma x : \gamma \in \Gamma\} \subset X$ is finite.
In~\cite[Proposition~5.1]{csc-ijm-2015}, it was observed that if $(X,\Gamma)$ is an expansive dynamical system whose periodic points are dense in  $X$,  then $(X,\Gamma)$ is surjunctive.
\par
Note that satisfying the d.c.c.\  is a hereditary property, in the sense that every subsystem of a  dynamical system satisfying the d.c.c.\ 
satisfies it as well. Expansivity is also hereditary 
but surjunctivity and density of periodic points are not. 
Consider for example the $\Z$-shift $(\{0,1\}^\Z,\Z)$
on the alphabet with two elements $0$ and $1$, 
and the  closed $\Z$-invariant subset  $\Sigma \subset \{0,1\}^\Z$ consisting of 
all bi-infinite sequences of $0$s and $1$s  with at most one chain of $1$s.
Then the continuous map $\tau \colon \Sigma \to \Sigma$ given by
$$
\tau(x)(n) =
\begin{cases}
1 & \text{if } (x(n),x(n+1)) = (0,1) \\
x(n) & \text{otherwise},
\end{cases}   
$$
is clearly  injective and commutes with the shift.
However, $\tau$ is not surjective since
a sequence of $0$s and $1$s where $1$ appears exactly once is in $\Sigma$ but not in the image 
of $\tau$.
Thus, the subsystem $(\Sigma,\Z)$ is not surjunctive.
On the other hand, 
periodic sequences are dense in $\{0,1\}^\Z$
while the only periodic sequences in $\Sigma$ are the two constant ones.
This example also shows that an expansive dynamical system may fail to be surjunctive.
\par
An \emph{algebraic dynamical system}
is a dynamical system of the form $(X, \Gamma)$, where $X$ is a compact metrizable topological group and $\Gamma$ is a countable group acting on $X$  by continuous group automorphisms.
Observe that if $S$ is any compact metrizable group, then the $\Gamma$-shift $(S^\Gamma,\Gamma)$ is an algebraic dynamical system for every countable group $\Gamma$. 
\par
Algebraic dynamical systems have been intensively studied in the last decades
(see the monograph by Schmidt~\cite{schmidt-book} and the references therein). As it was already observed by Halmos~\cite{halmos}  in the particular case $\Gamma = \Z$,
they provide a large supply of interesting examples for ergodic theory.
Moreover, Pontryagin duality can be fruitfully used when $X$ is abelian to analyse them.
\par
The goal of this note is to give sufficient conditions for an algebraic dynamical system to be surjunctive. 
 The first of our  results is the following.

\begin{theorem}
\label{t:th1}
Let $X$ be a  compact metrizable  group equipped with an action of a countable group $\Gamma$ by continuous group automorphisms. Suppose that $X$ is connected with finite topological dimension
and that the action of $\Gamma$ on $X$ is expansive. 
Then the dynamical system $(X,\Gamma)$ is surjunctive.
\end{theorem}

\begin{remark}
\label{r:lam}
By a result of Lam \cite[Theorem~3.2]{lam},
if a compact connected metrizable group $X$ admits an expansive action of a countable group 
$\Gamma$ by continuous group automorphisms,  then $X$ is necessarily abelian.
Thus, every group $X$ that satisfies  the hypotheses of 
 Theorem~\ref{t:th1} is abelian.
 Finite-dimensional connected metrizable abelian groups are often called \emph{solenoids} in the literature. By Pontryagin duality, solenoids are in one-to-one correspondance with finite-rank torsion-free abelian groups (see e.g.~\cite{morris}).
\end{remark}

\begin{theorem}
\label{t:th2}
Let $X$ be a compact metrizable group equipped with an action of  $\Z^d$ by continuous group automorphisms.
Suppose that the  action of $\Z^d$ on $X$ is expansive.
Then the dynamical system $(X,\Z^d)$ is surjunctive.
\end{theorem}

As every shift over a finite alphabet is expansive and any subsystem of an expansive system is itself expansive,
an immediate consequence of Theorem~\ref{t:th2} is the following.

\begin{corollary}
Let $S$ be a (possibly non-abelian) finite discrete group and let $X \subset S^{\Z^d}$ be a closed subgroup that is invariant under the $\Z^d$-shift.
Then the algebraic dynamical system $(X,\Z^d)$, where the action of $\Z^d$ on $X$ is the one induced by restriction of the $\Z^d$-shift,  is surjunctive.
\qed 
\end{corollary}

If  $(X,\Gamma)$ is an algebraic dynamical system, then the  Haar probability measure $\lambda_X$ on $X$ is $\Gamma$-invariant, i.e., $\lambda_X(\gamma A) = \lambda_X(A)$ for every measurable subset 
$A \subset X$ (this imediately follows from the uniqueness of Haar measure).
One says that $(X,\Gamma)$ is \emph{mixing} if
\begin{equation}
\label{e:def-mixing}
\lim\limits_{\gamma \to\infty}
\lambda_{X}(A_{1} \cap \gamma A_{2})
=\lambda_{X}(A_{1})\cdot\lambda_{X}(A_{2})
\end{equation}
for all measurable subsets $A_1,A_2 \subset X$ (here $\infty$ denotes the point at infinity in the 
one-point compactification of the discrete group $\Gamma$).
For example, the $\Gamma$-shift $(S^\Gamma,\Gamma)$ is mixing for any 
compact metrizable group $S$ and any countable group $\Gamma$  (just observe 
that~\eqref{e:def-mixing}  is obvious when 
$A_1, A_2 \subset S^\Gamma$ are cylinders).
\par
    One says that an algebraic  dynamical system $(X,\Gamma)$ satisfies the \emph{algebraic descending chain condition} (\emph{a.d.c.c.}\ for short)  if every decreasing sequence
$$
X = X_0 \supset X_1 \supset X_2 \supset \dots
$$
of $\Gamma$-invariant closed subgrops $X_i \subset X$ ($i \in \N$) eventually stabilizes.

\begin{theorem}
\label{t:th3}
Let $X$ be a compact metrizable abelian group equipped with an action of  $\Z^d$ by continuous group automorphisms.
Suppose that $X$ is connected and that  $(X,\Z^d)$ is mixing, satisfies the algebraic descending chain condition, and has finite topological entropy.
Then the dynamical system $(X,\Z^d)$ is surjunctive.
\end{theorem}

One can use Pontryagin duality to check that the hypotheses of Theorem~\ref{t:th3} are satisfied. 
Let  $X$ be a compact metrizable abelian group. Then its character group
$\widehat{X}$ is a discrete countable group.
Recall first that $X$ is connected if and only if $\widehat{X}$ is torsion-free.
When $X$ is equipped with a continous action of $\Gamma = \Z^d$, this action gives rise to a 
$\Z[\Gamma]$-module structure on $X$ and hence, by dualizing, to a $\Z[\Gamma]$-module structure on $\widehat{X}$. Here $\Z[\Gamma]$ denotes the group ring of $\Gamma = \Z^d$, which may be identified with the ring
$R_d := \Z[u_1,u_1^{-1},\dots,u_d,u_d^{-1}]$
 of Laurent polynomials with integral coefficients on $d$ commuting indeterminates.
 Then $(X,\Z^d)$ satisfies the a.d.c.c.\ if and only if $\widehat{X}$ is Noetherian as a $\Z[\Gamma]$-module. As the ring $R_d$ is Noetherian, this amount to saying that $\widehat{X}$ is a finitely generated module. Also, the fact that $(X,\Z^d)$ is mixing and has finite topological entropy can be read on   the $\Z[\Gamma]$-module structure of $\widehat{X}$
 (see~\cite{schmidt-book} and Example~\ref{ex:ledrappier-generalized} below).

\section{Proofs}

For the proof of Theorem~\ref{t:th1}, we shall use the following result.

\begin{lemma}
\label{l:inj-endo-surj}
Let $X$ be a compact connected metrizable abelian group with finite topological dimension.
Then every injective continuous group endomorphism of $X$ is surjective.
\end{lemma}

\begin{proof}
Let $\alpha \colon X \to X$ be an injective continuous group morphism and let $n$ denote the topological dimension of $X$.
Then the Pontryagin dual $\widehat{X}$ of $X$ is a discrete torsion-free group with rank $n$.
Therefore $V := \widehat{X} \otimes_\Z \Q$ is an $n$-dimensional $\Q$-vector space and
$\widehat{X}$ embeds as a subgroup of $V$ via the map $\phi \mapsto  \phi \otimes_\Z 1$.
Furthermore the dual endomorphism $\widehat{\alpha}$ uniquely extends to a $\Q$-linear map 
$\widehat{\alpha}_\Q \colon V \to V$.
The injectivity of $\alpha$ implies the
surjectivity of $\widehat{\alpha}$
(see~\cite[Proposition~30]{morris}) 
and hence the surjectivity of $\widehat{\alpha}_\Q$.
As every surjective endomorphism of a finite-dimensional vector space is injective,  it follows that
$\widehat{\alpha}_\Q$ is injective.
As $\widehat{\alpha}$ is the restriction of $\widehat{\alpha}_\Q$ to $\widehat{X}$, we deduce that 
$\widehat{\alpha}$ is itself injective and therefore we conclude that
$\alpha = \widehat{\widehat{\alpha}}$ is surjective.
  \end{proof}
  
\begin{proof}[Proof of Theorem~\ref{t:th1}]
As pointed out in the Introduction, the group $X$ is abelian by \cite[Theorem~3.2]{lam}.
Let $\tau \colon X \to X$ be an injective   $\Gamma$-equivariant continuous map.
It follows from  a rigidity result of Bhattacharya \cite[Corollary~1]{bhattacharya}  that $\tau$ is an affine map, i.e., there is a continuous group endomorphism
$\alpha  \colon X \to X$ and an element $b \in X$ such that
$\tau(x) = \alpha(x) + b$ for all $x \in X$.
The injectivity of $\tau$ is clearly equivalent to   that of $\alpha$.
Applying Lemma~\ref{l:inj-endo-surj}, we deduce that $\alpha$ is surjective.
This  implies that $\tau$ is surjective as well. 
\end{proof}

\begin{proof}[Proof of Theorem~\ref{t:th2}]
It follows from a result of Kitchens and Schmidt
\cite[Corollary~7.4]{kitchens-schmidt-1989} that, under theses hypotheses, the $\Gamma$-periodic points are dense in $X$.
On the other hand, as already mentioned in the Introduction,  it is known 
\cite[Proposition~5.1]{csc-ijm-2015} that every expansive dynamical system 
admitting a dense set of periodic points is surjunctive. 
\end{proof}

\begin{proof}[Proof of Theorem~\ref{t:th3}]
Let $\tau \colon X \to X$ be an injective   $\Gamma$-equivariant continuous map.
By a rigidity result due to Bhattacharya and Ward~\cite[Theorem~1.1]{bhattacharya-ward},
one has that $\tau$ is affine, that is, there is a continuous group endomorphism $\alpha \colon X \to X$ and an element
$b \in X$ such that  $\tau(x) = \alpha(x) + b$ for all $x \in X$. Since $\tau$ is injective, so is $\alpha$.
Consider now the sequence
\[
X = \alpha^0(X) \supset \alpha(X) = \alpha^1(X) \supset \alpha^2(X) \supset \cdots
\]
of $\Z^d$-invariant closed subgroups of $X$. 
Since $(X,\Z^d)$ satisfies the a.d.c.c., arguing as in the Introduction, we deduce that $X = \alpha(X)$.
It follows that $X = \alpha(X)+b = \tau(X)$, that is, $\tau$ is surjective.
\end{proof}

\section{Examples}

  \begin{example}
  The hypothesis that $X$ is finite-dimensional cannot be dropped from Lemma~\ref{l:inj-endo-surj}.
  Indeed,  $X := \T^\N$, where $\T := \R/\Z$ and $\N$ is the set of non-negative integers, is  a compact connected metrizable abelian group.
  However, the map $\alpha \colon X \to X$, given by
$\alpha(x)(0) := 0$ and $  \alpha(x)(n) := x(n - 1)$  for all $x \in X$ and $n \geq 1$,
   is a continuous injective group endomorphism of $X$ that is not surjective.
   Note that, for any countable group $mma$,  the $\Gamma$-shift $(X^\Gamma,\Gamma)$ yields an example of a non-surjunctive mixing algebraic dynamical system whose phase space 
is connected and abelian.
  \end{example}
  
 \begin{example}
    The hypothesis that $X$ is connected cannot be dropped from Lemma~\ref{l:inj-endo-surj}.
    Indeed, let $p$ be any prime and  thake $X := \Z_p$, the group of $p$-adic integers.
Then $X$ is a $0$-dimensional compact metrizable abelian group.    
However, the map $\alpha \colon X \to X$, defined by  $\alpha(x) := p x$ for all $x \in X$, is an injective continuous group endomorphism of $X$ that is not surjective.
Note that, for any countable group $\Gamma$,  the $\Gamma$-shift $(X^\Gamma,\Gamma)$ yields an example of a non-surjunctive mixing algebraic dynamical system whose phase space 
is $0$-dimensional and abelian.
\end{example}

\begin{example}
\label{ex:ledrappier}
Consider the $\Z^2$-shift $(\T^{\Z^2},\Z^2)$ and the closed $\Z^2$-invariant subgroup 
$X \subset \T^{\Z^2}$ consisting of all $x \colon  \Z^2 \to \T$ such that
$$
x(n_1 + 1,n_2) + x(n_1,n_2) + x(n_1,n_2 + 1) = 0_\T \quad \text{for all  } (n_1,n_2)   \in \Z^2.
$$
Observe that $X$ is connected. Indeed, consider the subset $E \subset \Z^2$ defined by
$$
E := \{(n_1,n_2) \in \Z^2 : n_2 = 0 \text{ or } (n_1 = 0 \text{ and } n_2 < 0)\}.
$$
Then the  restriction map $X \to \T^E$ is clearly a continuous group isomorphism.
Consequently, $X$  is isomorphic, as a topological group, to the infinite-dimensional torus $\T^\N$. 
The algebraic dynamical system $(X,\Z^2)$ is called the \emph{connected Ledrappier subshift}
(cf. \cite[Example~5.6]{lind-schmidt-handbook}, \cite{schmidt-book}).
It can be checked  that $(X,\Z^2)$ is mixing, satisfies the a.d.c.c., and has finite topological entropy
(see~\cite{schmidt-book}).
Therefore $(X,\Z^2)$ is surjunctive by Theorem~\ref{t:th3}.
Note that $(X,\Z^2)$ is not expansive (see~\cite{schmidt-book}) so that we could not have invoked Theorem~\ref{t:th2}
to get surjunctivity of $(X,\Z^2)$. 
\end{example}

\begin{example}
\label{ex:ledrappier-generalized}
The previous example may be generalized as follows.
Let $d \geq 1$ be an  integer and $f \in R_d = \Z[u_1,u_1^{-1}, \ldots, u_d,u_d^{-1}]$ an irreducible
(and hence non-zero) 
Laurent polynomial. Then
\[
f = \sum_{m \in \Z^d} f_m u^m,
\]
where $u^m := u_1^{m_1}\cdots u_d^{m_d}$, $f_m \in \Z$ for all
$m = (m_1,  \ldots, m_d) \in \Z^d$ and $f_m = 0$ for all but finitely many $m \in \Z^d$.
Consider the $\Z^d$-shift $(\T^{\Z^d}, \Z^d)$ and the closed $\Z^d$-invariant subgroup $X \subset \T^{\Z^d}$
consisting of all $x \colon \Z^d \to \T$ such that
\[
\sum_{m \in \Z^d} {f_m } x(n+m) = 0 \mbox{ for all } n \in \Z^d.
\]
Let $\C^* := \C \setminus \{0\}$ denote the non-zero  complex numbers and
$U \subset \C^*$ the unit circle.
\par
The \emph{variety} of $f$ is the subset $V_\C(f) \subset (\C^*)^d$ defined by
\[
V_\C(f) := \{c \in (\C^*)^d: f(c) = 0\}.
\]
Then the following holds (cf.\  \cite[Theorem 6.5 and Theorem 18.1]{schmidt-book}):
\begin{enumerate}[{\rm (1)}]
\item $\widehat{X}$ is a cyclic $R_d$-module isomorphic to $R_d/fR_d$;
\item $(X, \Z^d)$ satisfies the algebraic descending chain condition;
\item $(X, \Z^d)$ has finite topological entropy $h_{top}(X ,\Z^d) = \vert \log$\! M$(f) \vert$,
where $1 \leq $ M$(f) < \infty$ denotes the \emph{Mahler measure} of $f$;
\item the $\Z^d$-action on $X$ is mixing if and only if $f$ is not a generalized cyclotomic polynomial;
\item the $\Z^d$-action on $X$ is expansive if and only if $V_\C(f) \cap U^d = \varnothing$;
\item $X$ is connected if and only if $V_\C(f) \neq \varnothing$.
\end{enumerate}
Note that, in this context,
$$
X \text{ disconnected } \Rightarrow (X,\Z^d) \text{ expansive } \Rightarrow (X,\Z^d) \text{ mixing}. 
$$
\par
If $V_\C(f) = \varnothing$, then $(X,\Z^d)$ is expansive and
it follows from Theorem~\ref{t:th2} that
$(X,\Z^d)$ is surjunctive.
On the other hand, if
$V_\C(f) \not= \varnothing$ and $f$ is not a generalized cyclotomic polynomial,
then $X$ is connected and $(X,\Z^d)$ is mixing so that $(X,\Z^d)$
is surjunctive 
by  Theorem~\ref{t:th3}.
Consequently, $(X,\Z^d)$ is surjunctive as soon as $f$ is not a generalized cyclotomic polynomial.
\par
Let us give some examples of such systems. 
\par
First note (cf.\ \cite[Lemma 6.7]{schmidt-book}) that the space $X$ is disconnected, i.e., $V_\C(f) = \varnothing$, if and only if 
$f = \pm k u^m$ for some $k \in \Z \setminus \{0, 1,-1\}$ and $m \in \Z^d$. In that case, 
$(X,\Z^d)$ is conjugate to the $\Z^d$-shift $(S^{\Z^d}, \Z^d)$, where $S = \Z/k\Z$.
\par
If $d=1$ and $f = 3 - 2u_1$, then 
$$
X  = \{x \in \T^\Z: 3x(n) = 2x(n+1) \mbox{ for all } n \in \Z\}.
$$ This
is the \emph{Furstenberg subshift}.
Here $V_\C(f) = \{3/2\}$ so that $X$ is connected and $(X,\Z)$ is expansive.
Note that the evaluation map at $u_1 = 3/2$ induces a quotient map
$\widehat{X} = R_1/fR_1 \to \Z[1/6]$ that is an isomorphism of $R_1$-modules, where the $R_1$-module structure on $\Z[1/6]$ is the one generated by multiplication by $3/2$.    
\par
If $d=2$ and $f = 1 + u_1 + u_2$, 
then  $(X, \Z^2)$ is the connected Ledrappier subshift described in Example~\ref{ex:ledrappier}.
In that case, we have that
$$
V_\C(f) \cap U^2 = \{(\omega,\omega^2), (\omega^2, \omega)\},
$$ 
where $\omega = \exp(2\pi i /3)$), so that we recover the fact that $X$ is connected and  $(X,\Z^2)$ is mixing but not expansive.
\end{example}

\bibliographystyle{siam}

\begin{thebibliography}{10}

\bibitem{bhattacharya}
{\sc S.~Bhattacharya}, {\em Orbit equivalence and topological conjugacy of
  affine actions on compact abelian groups}, Monatsh. Math., 129 (2000),
  pp.~89--96.

\bibitem{bhattacharya-ward}
{\sc S.~Bhattacharya and T.~Ward}, {\em Finite entropy characterizes
  topological rigidity on connected groups}, Ergodic Theory Dynam. Systems, 25
  (2005), pp.~365--373.

\bibitem{csc-concrete}
{\sc T.~Ceccherini-Silberstein and M.~Coornaert}, {\em Surjunctivity and
  reversibility of cellular automata over concrete categories}, in Trends in
  harmonic analysis, vol.~3 of Springer INdAM Ser., Springer, Milan, 2013,
  pp.~91--133.

\bibitem{csc-ijm-2015}
\leavevmode\vrule height 2pt depth -1.6pt width 23pt, {\em Expansive actions of
  countable amenable groups, homoclinic pairs, and the {M}yhill property},
  Illinois J. Math., 59 (2015), pp.~597--621.

\bibitem{gottschalk}
{\sc W.~Gottschalk}, {\em Some general dynamical notions}, in Recent advances
  in topological dynamics ({P}roc. {C}onf. {T}opological {D}ynamics, {Y}ale
  {U}niv., {N}ew {H}aven, {C}onn., 1972; in honor of {G}ustav {A}rnold
  {H}edlund), Springer, Berlin, 1973, pp.~120--125. Lecture Notes in Math.,
  Vol. 318.

\bibitem{gromov-esav}
{\sc M.~Gromov}, {\em Endomorphisms of symbolic algebraic varieties}, J. Eur.
  Math. Soc. (JEMS), 1 (1999), pp.~109--197.
	
\bibitem{halmos}
{\sc P.R.~Halmos}, {\em On automorphisms of compact groups}, Bull. Amer. Math. Soc., 
49 (1943), pp.~619--624.

\bibitem{kitchens-schmidt-1989}
{\sc B.~Kitchens and K.~Schmidt}, {\em Automorphisms of compact groups},
  Ergodic Theory Dynam. Systems, 9 (1989), pp.~691--735.

\bibitem{lam}
{\sc P.-F. Lam}, {\em On expansive transformation groups}, Trans. Amer. Math.
  Soc., 150 (1970), pp.~131--138.

\bibitem{lind-schmidt-handbook}
{\sc D.~Lind and K.~Schmidt}, {\em Symbolic and algebraic dynamical systems},
  in Handbook of dynamical systems, {V}ol.\ 1{A}, North-Holland, Amsterdam,
  2002, pp.~765--812.

\bibitem{morris}
{\sc S.~A. Morris}, {\em Pontryagin duality and the structure of locally
  compact abelian groups}, Cambridge University Press, Cambridge-New
  York-Melbourne, 1977.
\newblock London Mathematical Society Lecture Note Series, No. 29.

\bibitem{schmidt-book}
{\sc K.~Schmidt}, {\em Dynamical systems of algebraic origin}, vol.~128 of
  Progress in Mathematics, Birkh\"auser Verlag, Basel, 1995.

\bibitem{weiss-sgds}
{\sc B.~Weiss}, {\em Sofic groups and dynamical systems}, Sankhy\=a Ser. A, 62
  (2000), pp.~350--359.
\newblock Ergodic theory and harmonic analysis (Mumbai, 1999).

\end{thebibliography}

\end{document}